\documentclass{article}
\usepackage{graphicx} 
\usepackage[T1]{fontenc}
\usepackage[utf8]{inputenc}
\usepackage[english]{babel}
\usepackage{amsmath}
\usepackage{bm}
\usepackage{amsthm}
\usepackage{amssymb}
\usepackage{enumitem} 
\usepackage{tikz}
\usetikzlibrary{matrix}
\usepackage{pgfplots}
\usepackage{graphicx}
\usepackage{faktor}
\usepackage{dsfont}
\usepackage{xcolor, colortbl}
\usepackage{makecell}
\usepackage{caption}
\usepackage{subcaption}
\usepackage{comment}
\usepackage{setspace}
\usepackage[12pt]{moresize}
\usepackage{wasysym}
\usepackage{float}
\usepackage[margin=1in]{geometry}
\usepackage{mathrsfs}
\usepackage{url}
\usepackage{authblk}
\usepackage{esint}
\usepackage{bbm}
\usepackage{hyperref}
\usepackage{cleveref}
\hypersetup{
	colorlinks = true,
	urlcolor   = blue,
	citecolor  = blue,
	linkcolor  = blue,
}
\usepackage{titlesec}

\titleformat{\subsection}[runin]
  {\normalfont\bfseries}
  {\thesubsection}
  {0.5em}
  {}
  [.]

\titlespacing{\subsection}
  {0pt}      
  {1em}      
  {0.8em}    

\let\Re\undefined

\DeclareMathOperator{\Re}{Re}

\DeclareMathOperator{\diver}{\mathrm{div} }

\newcommand{\R}{\mathbb{R}}

\newcommand{\Z}{\mathbb{Z}}
\newcommand{\N}{\mathbb{N}}

\newcommand{\cL}{{\mathcal L}}
\newcommand{\C}{{\mathbb{C}}}
\newcommand{\cH}{{\mathcal H}}
\newcommand{\cJ}{{\mathcal J}}
\newcommand{\sL}{{\mathscr L}}

\newcommand{\im}{\mathrm{i}\,}

\theoremstyle{plain}
\newtheorem{thm}{Theorem}[]

\newtheorem*{thm*}{Theorem}
\newtheorem{lem}[thm]{Lemma}

\newtheorem{cor}[thm]{Corollary}

\theoremstyle{remark}
\newtheorem{rmk}[thm]{Remark}

\numberwithin{equation}{section}

\title{\Large \textbf{LINEAR STABILITY OF THE LAMB-CHAPLYGIN DIPOLE}\\[0.5cm]}

\author{%
FRANCESCO PIO NUMERO AND PAOLO VENTURA\thanks{
Institute of Mathematics, EPFL, Station 8, 1015 Lausanne, Switzerland.
Email: \texttt{francesco.numero@epfl.ch, paolo.ventura@epfl.ch}.}
}

\date{}

\begin{document}
\maketitle

\vspace{-0.3cm} 

\noindent\textsc{Abstract.} We describe the linearized dynamics near the Lamb-Chaplygin dipole, a classical traveling solution of the two-dimensional Euler equations. Exploiting the Hamiltonian structure of the system together with its symmetries, we identify all possible sources of linear instability. For general perturbations in $L^1\cap L^p$, $p>2$, growth can occur through two explicit mechanisms triggered by: {\rm (i)} a nonzero circulation on the core of the dipole, and {\rm (ii)} a nontrivial component along the generalized eigenvectors associated with the eigenvalue $0$. In particular, we completely classify the spectrum and the Jordan chains of the operator associated with the linear dynamics. Both mechanisms hint for a nonlinear dynamics that may drift along the symmetry-generated family of traveling dipoles without moving away from it.

\section{Introduction}
We consider the two-dimensional Euler equation for an incompressible, inviscid fluid, 
in vorticity formulation 
\begin{equation} \label{Euler_equations}
\partial_{t} \omega + \mathrm{div} (u[\omega] \omega) = 0, \qquad t \geq 0, \, \, x=(x_1, x_2) \in \mathbb{R}^2,
\end{equation}
where $\diver (v_1,v_2) := \partial_{x_1}v_1 + \partial_{x_2}v_2$.
The velocity $u[\omega]$ of the fluid is determined by the usual Biot-Savart Fourier multiplier 
as follows
\begin{equation} \label{BiotSavart}
u[\omega] := \mathcal{F}^{-1}\left[-  \frac{\im\xi^{\perp}}{|\xi|^2} \mathcal{F}[\omega](\xi)\right],\quad \xi^\perp := (-\xi_2, \xi_1),
\end{equation}
where $\mathcal{F}$ denotes the Fourier transform  $\mathcal{F}[f] := \int_{\R^2}f(x)  e^{\im x \cdot \xi}  {\rm d}x $.

We are interested in the stability of \eqref{Euler_equations} near traveling solutions, namely profiles that evolve in time by pure translation and without changing shape. We thus introduce the traveling frame
\begin{equation*}
\begin{cases}
     x' = x + \mathbf{c} t, \quad \mathbf{c} \in \mathbb{R}^{2}, \\
     t' = t.
\end{cases}
\end{equation*}
In these variables, equation \eqref{Euler_equations} becomes
\begin{equation} \label{Euler_movingframe}
   \partial_{t} \omega = - \mathrm{div} ( (u[\omega] + \mathbf{c}) \omega)\, .
\end{equation}
It is a classical result of Yudovich \cite{Yudovich} that, for initial data
$\omega_0 \in L^1(\mathbb R^2)\cap L^\infty(\mathbb R^2)$, the Cauchy problem
for \eqref{Euler_equations} (equivalently for \eqref{Euler_movingframe}) admits a unique global solution in the same class. 
The existence argument can be extended to initial data in the broader $\omega_0 \in L^1(\mathbb R^2)\cap L^p(\mathbb R^2) $, $p>2$.

A well-known steady solution of \eqref{Euler_movingframe} is the Lamb-Chaplygin dipole, given by --in polar coordinates-- 
\begin{equation}\label{Lambdipole}
    \omega_{eq}(r,\theta) = \begin{cases}
        A J_{1}(kr) \sin(\theta) \quad {\rm{in}} \ {\rm D}, \\
        0 \quad {\rm{in}} \  \mathbb{R}^{2} \setminus {\rm D},
    \end{cases} \quad {\bf c}_{eq} := (1, 0),
\end{equation}
where ${\rm D}:=\{ x \in \R^2 : |x| \leq 1\}$ is the unit disk and $A := \frac{2 k}{J_{0}(k)}$. Here $J_i=J_i(r)$, $i\in\{0,1,\dots\}$, are the Bessel functions of the first kind and $k=3.8317\dots$ is the first nontrivial zero of $J_1$. We observe that, inside the unit disk $\rm D$, one has the functional relation
\begin{equation}
    \label{functionalrelation}
\omega_{eq} = - k^2 \psi_{eq},
\end{equation}
where $\psi_{eq}$ is the stream function of the Lamb-Chaplygin dipole, determined (up to constants) by the relation 
\begin{equation}\label{velocity}
    \nabla^\perp \psi_{eq} = u_{eq},\quad u_{eq}:= u[\omega_{eq}] +{\bf c}_{eq},
\end{equation}
with $u_{eq}$ being the  physical velocity field of the dipole. In formulas
\begin{equation}
    \psi_{eq}(r,\theta) = \begin{cases}
        -\frac{A}{k^2} J_{1}(kr) \sin(\theta) \quad r \leq 1 \\
        (\frac{1}{r}-r ) \sin(\theta) \quad r > 1.
    \end{cases}
\end{equation}
\begin{figure}
    \centering    \includegraphics[width=0.7\linewidth]{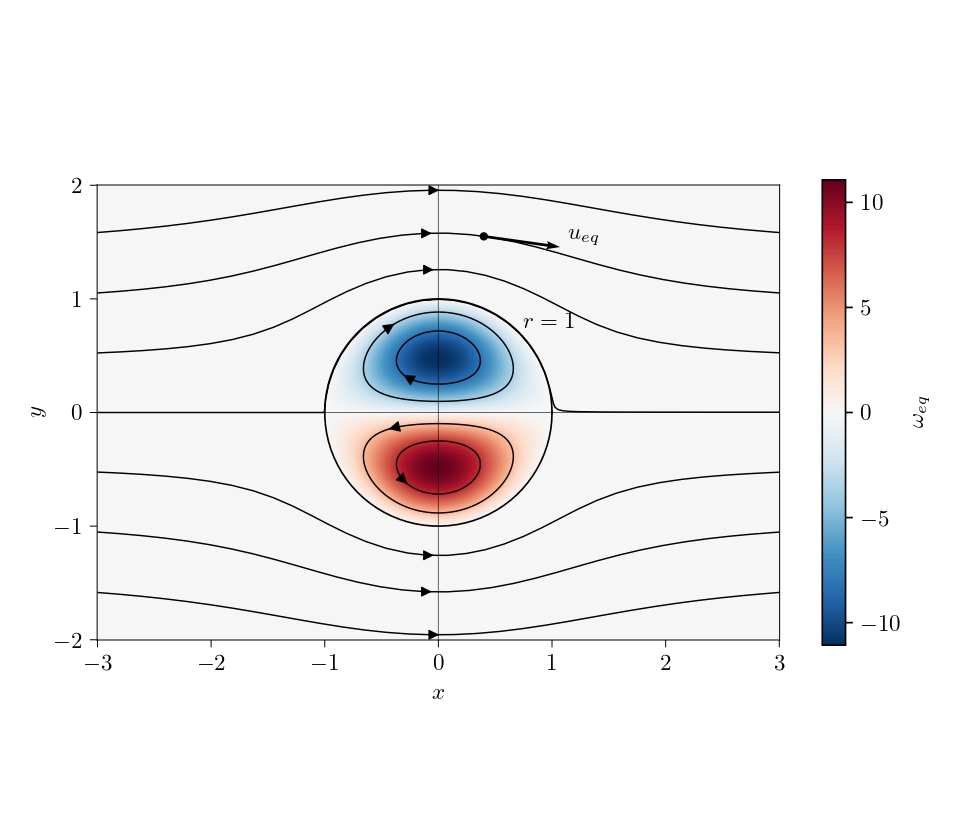}
          \caption{Streamlines of the Lamb-Chaplygin dipole shown in black. The red and blue regions represent the positive and negative vorticity intensities, respectively.}
    \label{fig:Lambdipole}
\end{figure}
The Lamb--Chaplygin dipole introduced in \eqref{Lambdipole} traces back to
the classical works of Chaplygin \cite{Chaplygin1903} and Lamb
\cite{Lamb1906}; see also \cite{MeleshkoVanHeijst1994} for a historical account. More recently, it has attracted renewed attention by the fluid community.
On one hand, its orbital stability was proved by Abe and Choi
\cite{AbeChoi2022} under symmetry, sign and moment assumptions, and later refined
in \cite{AbeChoiJeong2025,LiSongZhou2026}. On the other hand, it has gained a primary role
as building block in the sharp convex-integration construction of
Bru\'e, Colombo, and Kumar \cite{BrueColomboKumar2024}.

A natural question is whether the Lamb-Chaplygin dipole is stable under
arbitrary perturbations in the class $w_0 \in L^1 \cap L^p$, $p>2$, with the distance between the perturbed
solution and the steady state measured in the $L^2$ norm. Hereafter, we denote
\begin{equation*}
    \| w\|_{L^2} := \Big( \int_{\R^2} |w(x)|^2 {\rm d}x\Big)^{\frac12},\quad \| w\|_{L^1\cap L^p} :=  \int_{\R^2} |w(x)| {\rm d}x + \Big( \int_{\R^2} |w(x)|^p {\rm d}x\Big)^{\frac1p} .
\end{equation*}
The present paper takes a step toward answering the latter question, which, to the best of our knowledge, remains open in such a general setting.

As usual, our starting point in the stability analysis of system \eqref{Euler_movingframe} is the study of the linear evolution of the initial datum $\omega_{eq} + \epsilon w_0$, with $\epsilon$ sufficiently small to discard the terms of order $O(\epsilon^2)$. 
By linearizing \eqref{Euler_movingframe} at its steady solution $ (\omega_{eq}, {\bf c}_{eq})$ one obtains the linear system

\begin{equation}\label{cLeq}
 \begin{cases}
        \partial_t w(t) = \cL_{eq} w(t), \\
        w(0) = w_0  
        ,
    \end{cases}\quad    \cL_{eq} w := - \diver( u_{eq} w) - \diver(u[w] \omega_{eq}),
\end{equation}
with $\omega_{eq}$ in \eqref{Lambdipole} and $u_{eq}$ in \eqref{velocity}.

Our first main result  is the following
\begin{thm}\label{lineardyn_Thm} 
    Let $w_0 \in L^1(\R^2)\cap L^p(\R^2)$, $p>2$, and  $\mu := |{\rm D}|^{-1}\int_{\rm D} w_0$ be the average of $w_0$ on the unit disk.  The solution $w\in C^0\big(\R, L^1(\R^2) \cap L^p(\R^2)\big) $ of the linear problem \eqref{cLeq} 
    grows at most quadratically in time: there exist constants $C>c>0$, independent from $w_0$, such that
\begin{equation} \label{quadraticgrowth}
         \| w(t) \|_{L^2} \leq  \| w_0\|_{L^2} +  C \big[ \| w_0\|_{L^2} t +  |\mu| t^2  + \| w_0\raisebox{-0.7ex}{\big|}_{\R^2\setminus {\rm D}} \|_{L^1\cap L^p} ( t + t^2) \big] ,\ \ \text{ for every }t>0,
    \end{equation}
and, if the initial datum $w_0$ is supported on the unit disk ${\rm D}$,
\begin{equation}\label{quadraticgrowthinD}
     \| w(t) \|_{L^2} \geq c |\mu| t^2  - C \| w_0\|_{L^2} t .
\end{equation}
Moreover, there exist initial data $w_0$
with average $\mu
= 0$, for which $\|w(t)\|_{L^2}$ grows linearly in time.
\end{thm}
For an arbitrary generator $\cL$ a precise polynomial-growth estimate is far from being trivial. On one hand, isolated eigenvalues with positive real part provide an immediate mechanism for exponential growth of the semigroup $e^{t\cL}$. On the other hand, the absence of such eigenvalues, or even more detailed spectral information, is in general not sufficient to control the semigroup evolution:
 Renardy's PDE counterexample shows that, even for a hyperbolic equation, linear stability cannot in general be inferred from the spectrum alone \cite{Renardy1994}. The separation between spectral information and semigroup growth is already apparent in the celebrated Gearhart--Pr\"uss theorem on Hilbert spaces \cite{Gearhart1978,Pruss1984} (see \cite{DellOroSeifert2022} for a short elementary proof). Their result shows that subexponential growth of the semigroup is equivalent to a uniform resolvent estimate along the imaginary axis, a requirement much stronger than merely locating the spectrum in the closed left half-plane.
Here we conclude much more than a subexponential growth estimate: we provide a  quadratic-in-time bound. Moreover, the starting functional setting of our analysis is not Hilbertian. As observed in \cite{Davies2005}, the linear evolution of a semigroup may behave very differently in $L^1$ and in $L^2$, even when the spectrum of the generator is unchanged.

The derivation of estimate \eqref{quadraticgrowth} relies on a precise description of the linearized dynamics. 
Since the vorticity $\omega_{eq}$ of the Lamb-Chaplygin dipole is supported on the unit disk, the linear evolution in \eqref{cLeq} separates neatly into two parts.
Outside the unit disk the perturbation is simply transported by the dipole
flow. Inside the unit disk ${\rm D}$, the initial datum evolves according to a forced linear system:
its homogeneous part is governed by a closed operator
$\sL : {\rm Dom}(\sL)\subseteq L^2_0({\rm D}) \to L^2_0({\rm D})$,
while the forcing term is due to the average of $w_0$ in ${\rm D}$ and to the outer part of the evolution. Here $L^2_0({\rm D})$ is the Hilbert space of square-integrable functions on  ${\rm D}$ with zero average and
\begin{equation} \label{Linear_operator}
    \mathscr{L} := - (u_{eq} \cdot \nabla) ({\rm Id} + k^2 \Delta_{D\! N}^{-1}),
\end{equation}
where $\Delta_{D\! N}^{-1}$ is the inverse Laplace operator with suitable boundary conditions on $\partial {\rm D}$ descending from the analysis of the evolution outside the disk, see Section \ref{passingtosL}. We will be able to convert 
the spectral properties
of the operator $\sL$ in \eqref{Linear_operator}, including its Jordan chains, into quantitative information on the dynamics of \eqref{cLeq}. The next theorem provides a complete description of this spectral structure.

Before presenting the statement of the theorem, let us recall that a Jordan chain of length $m\geq 2$ associated with an eigenvalue $\lambda$ of $\sL$ is a set of the form
\begin{equation}
    \label{Jordanchain}
\{ f_{m-1}, f_{m-2}, \dots, f_0 \} \subset {\rm Dom}(\sL),
\end{equation}
with $\sL f_0 = \lambda f_0 $, $f_0 \neq 0$, and $\sL f_{j} = \lambda f_j + f_{j-1} $ for every $j=1,\dots,m-1$. We also recall that, if $\Re(\lambda) \geq 0$, such a structure entails a  growth mechanism for the linear system $\dot w = \sL w$. For example, if $\{ f_1, f_0\}$ is a Jordan chain of length $2$ associated with $\lambda$, then $w(t) = e^{t \lambda}( f_1 + t f_0)$ is a growing solution with initial datum $w(0) = f_1 $.

Our second main result is the following 
\begin{thm} \label{Spectralstability_Thm} Let $\sL$ be the operator introduced in  \eqref{Linear_operator}. Then:
\begin{enumerate}
    \item \label{primopunto} $\sL$ is spectrally stable, namely
$
\sigma_{L^2_0({\rm D})}(\sL) \subseteq \im \R
$, and each of its nonzero eigenvalues is semisimple.
\item  \label{secondopunto} $\sL$ is linearly unstable. It has exactly two (up to linear combinations) Jordan chains of length $2$ associated with $0$.
\end{enumerate}
\end{thm}
Let us comment on Theorems \ref{lineardyn_Thm}-\ref{Spectralstability_Thm}.
\begin{enumerate}[leftmargin=0pt,
    labelsep=0.5em,
    itemindent=2em]
    \item We are able to characterize completely the linear dynamics near the Lamb-Chaplygin dipole, without restricting ourselves to any particular symmetry class of initial data as done in prior literature. The estimates that we show in \eqref{quadraticgrowth}-\eqref{quadraticgrowthinD} are sharp, in the sense that the only sources of growth, for a general initial datum $w_0$ supported on the unit disk ${\rm D}$ and evolving according to the linear system \eqref{cLeq}, are:
    \begin{itemize}
        \item a nonzero average $\mu \neq 0$ on ${\rm D}$, which gives rise to a quadratic-in-time detachment from the original equilibrium;
        \item a nontrivial component along one of the two Jordan chains in the generalized kernel of $\sL$, yielding a linear-in-time drift.
    \end{itemize} 
    \item The two mechanisms of linear instability that we have shown above are direct consequences of the symmetries preserved by Euler equations. Indeed, as we will show in Section \ref{symmetries}, the Lamb-Chaplygin dipole appears inside a continuum of traveling solutions. The growth reflects the possibility for an initial datum to evolve by drifting along this ensemble of steady states. In particular, as we prove in Lemma \ref{generalizedkernel}, the unstable generalized eigenvectors correspond to infinitesimal  rotations/amplifications of the Lamb dipole and generate spatial translations of the solution. The latter will then move away from the original steady state at constant speed, while keeping invariant distance with respect to the manifold of Lamb dipoles. On the other hand, as we show in the end of Section \ref{symmetries}, a nontrivial average on the unit disk gives rise, under the evolution of the full Euler equations, to an explicit rototranslation of the Lamb dipole. The linearized dynamics retains a trace of this behavior, as reflected in the quadratic-in-time growth displayed in \eqref{quadraticgrowth}.
    \item Our linear analysis 
fits naturally with the previous stability results for the Lamb--Chaplygin
dipole by clarifying the linear mechanisms underlying them. In particular, the
odd-symmetry assumption in \cite{AbeChoi2022,AbeChoiJeong2025} eliminates both
the average over the unit disk and the generalized unstable direction associated with infinitesimal rotations, namely  the obstructions whose absence underlies the
orbital stability mechanism modulo horizontal shifts developed in those works.\\
The spectral-stability result recently appeared in
\cite{LiSongZhou2026}, and comparable with point \ref{primopunto} of Theorem \ref{Spectralstability_Thm}, moves in a similar direction.
 In that paper, the
authors study the spectrum of the linearized operator $\cL_{eq}$ in \eqref{cLeq} directly and
prove the absence of unstable eigenvalues by relying on the general index theory
for Hamiltonian PDEs developed by Lin and Zeng \cite{LinZeng2022}. 
This powerful
framework, designed for a broad class of Hamiltonian problems, is implemented after realizing the operator on suitable weighted
subspaces of $L^2(\R^2)$, which provide the additional decay needed to handle
the Biot--Savart operator on the whole plane; see also
\cite[Remark B.4.1]{afterVishik}.\\
We also mention the recent work of Wang \cite{Wang2025Disk}, where a truncated
version of the Lamb--Chaplygin dipole is studied as a steady solution of the
Euler equation in a fixed disk, and orbital stability is proved by variational
methods. In Wang's disk setting, the truncation fixes the vortical region and removes the modulation mechanism associated with changes of the amplitude of the full Lamb--Chaplygin family.
\end{enumerate}

Let us outline the strategy of our proof. 

We begin by obtaining, in Section \ref{symmetries}, the generalized eigenvectors of the operator $\cL_{eq}$ in \eqref{cLeq} associated with $0$ by differentiating the symmetries of system \eqref{Euler_movingframe}. This is the standard mechanism
by which continuous families of equilibria generate neutral directions for the linearized operator. In the present traveling-wave setting, however, some symmetries also change the traveling velocity of the Lamb--Chaplygin dipole; differentiating along these parameters therefore produces generalized eigenvectors, rather than genuine kernel elements. This is exactly the strategy followed in \cite{NguyenStrauss2023} to obtain, starting from the underlying symmetries of the water waves system, the Jordan chains associated with $0$ for Stokes waves that play an essential role also in the proof of the long-wave spectral picture in \cite{BertiMasperoVentura2022,
BertiMasperoVentura2023}.

The viewpoint that we adopt in Sections \ref{sec:cuore} is dynamical rather than
purely spectral, in contrast with \cite{LiSongZhou2026}. We do not treat $\cL_{eq}$ as an operator to be analyzed directly
through its spectrum, since, as observed by the authors of the latter paper, in the natural $L^2$-setting this operator is not
closed. Instead, we interpret $\cL_{eq}$ as the generator of a strongly continuous
semigroup and consider the linear evolution associated with
arbitrary initial data in the vorticity class
$L^1(\R^2)\cap L^p(\R^2)$, $p>2$. Exploiting the compact support of the
vorticity of the Lamb--Chaplygin dipole, we reduce the relevant part of the
dynamics to a suitable $L^2$-closed operator $\sL$ on the unit disk. This approach avoids
the use of weighted $L^2$ spaces and, at the same time, isolates the mild sources of growth of the linear evolution.

The rest of our proof, presented in Section \ref{sec:Hamiltonian}, relies entirely on the Hamiltonianity of the operator $\sL= \cJ \cH$, cfr.\ \eqref{Hamiltonian}, which is itself a remnant of the
Hamiltonian structure of the two-dimensional Euler equations, see \cite{Olver1982}. This point of
view is in line with the previous stability analyses in \cite{AbeChoi2022,AbeChoiJeong2025,LiSongZhou2026}, 
and has been fully exploited in \cite{InstabilityCellularFlow} to study the stability of Taylor-Green vortices. The key point is that the  generalized eigenvectors presented in Section \ref{symmetries} already
exhaust the maximal number of generalized unstable directions allowed by the Hamiltonian structure. This rules out both eigenvalues with positive real part
and nontrivial Jordan chains associated with nonzero imaginary eigenvalues. The only substantial computation in the proof of Theorem \ref{Spectralstability_Thm} is the verification of this last point, carried
out in Lemmata \ref{stability_condition} and \ref{toSpain}. 

Section \ref{sec:linearstability} combines the spectral information obtained in the proof of Theorem \ref{Spectralstability_Thm} with the dynamical argument needed to prove Theorem \ref{lineardyn_Thm}. The main difficulty is that the standard Hamiltonian energy argument cannot be applied directly. Such an argument would require embedding ${\rm Ran}(\sL)$ in an $\sL$-invariant  closed subspace $Y$ on which $\cH$ is coercive.
One can define $Y$ in a standard way as the $L^2$-orthogonal to the Jordan chains of $\sL^*=-\cH \cJ$ constructed as the image through $\cH$ of those of $\sL$. In our problem, one readily finds {\it one} Jordan chain for $\sL^*$, in contrast with the {\it two} obtained for $\sL$ in Section \ref{symmetries}, see Remark \ref{rmk:dualJordanchain}. This is a usual obstruction in the non-normal infinite-dimensional setting, where the Jordan structure of an operator need not mirror that of its adjoint. Consequently, $\cH$ is not coercive on $Y$ and the standard stability argument does not apply.

Nevertheless, we will be able to show that $\cH$ is positive definite on the non-closed subspace ${\rm Ran}(\sL^2)$. This is still sufficient to conclude that the evolution of every initial datum in ${\rm Ran}(\sL)$ remains bounded in time. As a consequence, a general initial datum in $L^2_0({\rm D})$ can grow at most linearly in time. This completes the proof of Theorem \ref{lineardyn_Thm}.
\smallskip

We now pass to the analysis of the symmetries of
\eqref{Euler_movingframe} and of the Jordan chains generated by them.

\section{The family of Lamb-Chaplygin dipoles} \label{symmetries}
The  Lamb-Chaplygin dipole presented in \eqref{Lambdipole} is not an isolated traveling solution of \eqref{Euler_equations}, since it belongs to a five-parameter family of steady states of \eqref{Euler_movingframe} arising from the symmetries of the system.  In this section we present the construction of this family of steady solutions of \eqref{Euler_movingframe} and  its aftermaths on the linear operator $\cL_{eq}$ in \eqref{cLeq}.

{Let us define the vector field $F(\omega, \mathbf{c}) := - \mathrm{div} ( (u[\omega] + \mathbf{c}) \omega)\,$, corresponding to the right hand side of \eqref{Euler_movingframe}. As one can readily see by inspection, the function $F$ in \eqref{Euler_movingframe} is invariant under the following transformations}, for every $\alpha \in \R$:
    \begin{enumerate}[label=(\roman*)]
        \item \label{traslainx1} {\bf horizontal translation}: $F( \omega(x_1+\alpha,x_2), \mathbf{c}) =  F(\omega, \mathbf{c})(x_1+\alpha,x_2)$;
        \item {\bf vertical translation}: $F(\omega(x_1,x_2+\alpha), \mathbf{c}) =  F(\omega, \mathbf{c})(x_1,x_2+\alpha)$;
        \item\label{rotation} {\bf rotation}: $F(\omega( R_{\alpha} x), R^{T}_{\alpha} \mathbf{c}) =  F(\omega, \mathbf{c})(R_{\alpha} x)$ for every rotation $R_{\alpha} := \begin{pmatrix} \cos(\alpha) & -\sin(\alpha) \\ \sin(\alpha) &  \cos(\alpha) \end{pmatrix}$;
        \item {\bf amplification}: $F(\alpha \omega( x), \alpha \mathbf{c}) = \alpha^{2} F(\omega, \mathbf{c})(\alpha x)$;
        \item \label{dilation} {\bf dilation}: $F( \omega(\alpha x), \alpha^{-1} \mathbf{c}) =  F(\omega, \mathbf{c})(\alpha x)$, if $\alpha \neq 0$.
    \end{enumerate}
{When applied to the Lamb-Chaplygin dipole, horizontal and vertical translations represent rigid spatial shifts along the $x_1$ and $x_2$ axes, respectively, leaving the internal structure, size, and translation velocity ${\rm {\bf c}}$ unaffected. On the other hand, symmetries \ref{rotation}-\ref{dilation} change the traveling velocity of the dipole. More precisely, a rigid rotation pivots the geometric structure and rotates the direction of propagation of the the dipole by the same angle, whereas the dipole propagates $\alpha$ times faster across the plane under an $\alpha$-amplification of the vorticity. Finally, a spatial dilation as in \ref{dilation} induces an inverse $\alpha^{-1}$ rescaling of the traveling velocity. 

To characterize the entire family of equilibria, we introduce the following notation.} 
Let $\mathcal{G}$ be the group of transformations generated by arbitrary compositions of the operations \ref{traslainx1}-\ref{dilation}. Any element $\tau \in \mathcal{G}$ acting on the steady state $(\omega, \bf c)$ can be uniquely represented by a set of parameters $\bm{\alpha} = (\alpha_1, \dots, \alpha_5)$, $\alpha_5 \neq 0$, such that
\begin{equation}
    \tau_{\bm{\alpha}}\omega(x) := \alpha_4 \, \omega\left( \alpha_5 R_{\alpha_3} x + \begin{pmatrix} \alpha_1 \\ \alpha_2 \end{pmatrix} \right), \qquad
    \gamma_{\bm{\alpha}}{\bf c} := \frac{\alpha_4}{\alpha_5} R_{\alpha_3}^T \mathbf{c}.
\end{equation}
It is straightforward to notice that if $F(\omega, {\bf c}) = 0$, i.e. $(\omega, \bf c)$ is an equilibrium, then $F(\tau_{\bm \alpha} \omega, \gamma_{\bm{\alpha}}{\bf c})=0$. This property naturally establishes a continuous family of equilibria whose elements will simply be referred to as Lamb-Chaplygin dipoles in the sequel.

By linearizing \eqref{Euler_movingframe} at its steady solution $(\tau_{\bm \alpha} \omega_{eq}, \gamma_{{\bm \alpha}}{\bf c}_{eq})$ one obtains the linear system
\begin{equation}\label{linearsystemintro}
 \partial_t w =   \cL_{\bm \alpha} w \,,\quad     \cL_{\bm \alpha} w = - \diver (u_{{\bm \alpha}}  w) - \diver(u[w] \omega_{{\bm \alpha}})\, .
\end{equation}
where $\omega_{\bm \alpha} :=\tau_{\bm \alpha} \omega_{eq}$ and $u_{\bm \alpha} := u[\tau_{\bm \alpha}] + \gamma_{\bm \alpha} {\bf c}$.
It is straightforward to observe that 
\begin{equation}\label{conjugation}
    \cL_{\bm \alpha} := \alpha_4^2\; \tau_{\bm \alpha} \circ \cL_{eq} \circ  \tau_{\bm \alpha}^{-1},
\end{equation}
with $\cL_{eq}$ in \eqref{cLeq}.
In view of \eqref{conjugation} the linear evolution near any Lamb-Chaplygin dipole is completely determined by studying the semigroup generated by $\cL_{eq}$.

{We now present the elements of the generalized kernel of $\cL_{eq}$ associated with symmetries \ref{traslainx1}-\ref{dilation}}.
\begin{lem} \label{generalizedkernel}
The following functions
\begin{equation}\label{f1oddev}
    f_{1}^{\rm odd} := \frac{2k}{J_{0}(k)} J_{1}(kr) \sin(\theta)  {\bf 1}_{{\rm D}}, \quad 
    f_{1}^{\rm ev} := \frac{2k}{J_{0}(k)} J_{1}(kr) \cos(\theta) {\bf 1}_{{\rm D}},
\end{equation}
fulfill 
\begin{equation}\label{actionJordan}
     \cL_{eq} \, f_{1}^{\rm odd} = f_{0}^{\rm odd}, \quad \cL_{eq} \, f_{1}^{\rm ev} = -f_{0}^{\rm ev}, 
\end{equation}
where $f_0^{\rm odd}$ and $f_0^{\rm ev}$ are the following  kernel elements of $\cL_{eq}$
\begin{equation}\label{f0oddev}
    f_{0}^{\rm odd} := J_{2}(kr) \sin(2\theta) {\bf 1}_{{\rm D}}, \quad 
    f_{0}^{\rm ev} := \frac{ k^2}{J_{0}(k)} \big[ J_0(kr) + J_{2}(kr) \cos(2\theta) \big]{\bf 1}_{{\rm D}} \, . 
\end{equation}
\end{lem}
\begin{proof}
We will use the family of symmetries introduced in Section \ref{symmetries}. 
Let $s$ be a real number parameterizing one of  those invariant transformations, and let $(\omega_s, \mathbf{c}_s)$ be the related Lamb-Chaplygin dipole. By differentiating $F(\omega_s, \mathbf{c}_s) = 0$ with respect to $s$ at  $s=\Bar{s}$, where $\omega_{\Bar{s}} = \omega_{eq}$, one obtains the identity
\begin{equation}
    D_{\omega}F(\omega_{eq},{\bf c}_{eq})[\left. \partial_s \omega_s \right|_{s=\Bar{s}}] + D_{\mathbf{c}} F(\omega_{eq},{\bf c}_{eq}) [\left. \partial_s \mathbf{c}_s \right|_{s=\Bar{s}}] = 0 ,
\end{equation}
namely
\begin{equation}
\cL_{eq}(\left. \partial_s \omega_s \right|_{s=\Bar{s}}) = \left. \partial_s \mathbf{c}_s \right|_{s=\Bar{s}} \cdot \nabla \omega_{eq},
\end{equation}
where $\cL_{eq}$ is the linearized operator in \eqref{cLeq}.
For horizontal and vertical translations we have $\Bar{s}=0$ and
\begin{align}
     \cL_{eq} \, f_0^{\rm odd}  = {\cL}_{eq} \, f_0^{\rm ev} = 0,\quad \text{ where } f_0^{\rm odd} = \partial_{x_1} \omega_{eq}, \quad f_0^{\rm ev} =\partial_{x_2} \omega_{eq}.
\end{align}
Indeed, in polar coordinates, 
\begin{equation*}
    \partial_{x_1} \omega_{eq} = A \sin(\theta) \cos(\theta) \left(J_{1}^{'}(kr) k - \frac{1}{r} J_{1}(kr) \right),
\end{equation*}
and
\begin{equation*}
    \partial_{x_2} \omega_{eq} = A J_{1}^{'}(kr) k \sin^{2}(\theta) + \frac{1}{r} A J_{1}(kr) \cos^{2}(\theta).
\end{equation*}
By classical trigonometric identities and recurrence formulas for Bessel functions in \cite[Appendix~B, Section~B.2]{grafakos2008}, one obtains the explicit formulas \eqref{f0oddev}.
For the amplitude rescaling we have $\Bar{s}=1$ and
\begin{equation*}
    \cL_{eq} \, f_1^{\rm odd} = \mathbf{c}_{eq} \cdot \nabla \omega_{eq} = \partial_{x_1} \omega_{eq} = f_{0}^{\rm odd}, \text{ where } f_1^{\rm odd}:=\omega_{eq}\,,
\end{equation*}
whence the formula for $f_1^{\rm odd}$ in \eqref{f1oddev} descends from \eqref{Lambdipole}.
Finally, for the rigid rotation we have $\Bar{s}=0$ and
\begin{equation*}
 \cL_{eq} \, f_1^{\rm ev} =-(\mathbf{c}_{eq}^{\perp} \cdot \nabla) \omega_{eq} = - \partial_{x_2} \omega_{eq} = - f_0^{\rm ev},\text{ where }f_1^{\rm ev}:=\partial_{\theta} \omega_{eq},
\end{equation*}
and one obtains the explicit formula for $f_1^{\rm ev}$ in \eqref{f1oddev}. 
\end{proof}
\begin{rmk}
    By considering the dilation symmetry in \ref{dilation} and  differentiating $F(\omega_{s},{\bf c}_{s})=0 $ with respect to $s$, one obtains another kernel function for $\cL_{eq}$, namely $rJ_0(kr)\sin(\theta) {\bf 1}_{\rm D}$. We also observe that the unstable directions established in Lemma \ref{generalizedkernel} have their support inside the unit disk.
\end{rmk}

We conclude this section by pointing out that the source of quadratic growth in \eqref{quadraticgrowthinD}, namely the average of the initial datum $w_0$ on the unit disk, is an artifact of linearizing around the single equilibrium $\omega_{eq}$ rather than considering stability with respect to the whole family of traveling solutions. Indeed, if one takes
as initial datum for the full Euler equations in  \eqref{Euler_equations}
$$
\omega_{in}:=\omega_{eq}+\mu \mathbf 1_D,
\quad  \mu \neq 0,
$$
with initial traveling velocity $ {\bf c}_{eq}=(1,0) $,
then the corresponding solution is given exactly by
$$
\omega(t)
=
\tau_{\bm\alpha(t)}(\omega_{eq}+\mu \mathbf 1_D),\quad \bm\alpha(t)
:=
\big(-\tfrac{2}{\mu}\sin\big(\tfrac{\mu t}{2}\big),
\tfrac{2}{\mu}\big(
1-\cos\big(\tfrac{\mu t}{2}\big)
\big),-\tfrac{\mu t}{2},1,1 \big) ,
$$
with traveling velocity 
$
{\bf c}(t) =
\gamma_{\bm\alpha(t)}{\bf c}_{eq} =
R_{\mu t/2}{\bf c}_{eq}.
$
In other words, the perturbation $\mu \mathbf 1_D$ does not generate a
genuine instability of the nonlinear dynamics: it merely adds a rigid
rotation inside the disk, so that the solution remains for all times on the
orbit of the family of traveling Lamb--Chaplygin dipoles.\smallskip 

We now turn our attention to the dynamics generated by the linearized flow \eqref{cLeq}.

\section{Linear dynamics near the Lamb-Chaplygin dipole}\label{sec:cuore}
In this section, we continue our study of the linear system in \eqref{cLeq} by shifting the focus from the spectral properties of $\cL_{eq}$ to the reduction of the problem to the internal dynamics of the disk. This leads to the definition of the operator $\sL$ introduced in \eqref{Linear_operator}. The analysis of its spectral properties, together with an abstract argument, culminates in the proof of Theorem \ref{Spectralstability_Thm}. The spectral picture will then be combined with a dynamical argument to obtain suitable growth estimate and complete the proof of Theorem~\ref{lineardyn_Thm}.

\subsection{Reduction to the unit disk}\label{reductiontodisk}

Let us consider the system \eqref{cLeq}. The first term of $\cL_{eq}$ is the classical transport operator which generates a strongly continuous semigroup on $X_p = L^1(\R^2)\cap L^p(\R^2)$, $p>2$, whereas the second one is a bounded operator on the same class. Indeed, for every $w\in X_p$ one has $u[w] \in L^{\infty}(\R^2)$ (cfr. \cite[Remark 1.0.2]{afterVishik}) and, by inspection, $\nabla \omega_{eq} \in L^{1}(\R^2) \cap L^{\infty}(\R^2)$.
By  \cite[Chapter III, Theorem 1.3]{EngelNagel} we conclude that $\cL_{eq}$ generates a strongly continuous semigroup on $X_p$.

To study the linear stability of the Lamb-Chaplygin dipole, we first observe that the dynamics can be rigorously restricted to the unit disk ${\rm D}$. We split a solution $w(t)\in X_p$ of $ \partial_t w(t) = \cL_{eq} \, w(t)$ into $w(t) = w_{ext}(t) + w_{int}(t)$. Here we define $w_{ext}$ as the solution of
\begin{equation} \label{wext}
    \begin{cases}
        \partial_{t} w_{ext}(t)= -u_{eq} \cdot \nabla w_{ext}(t), \\
        w_{ext}(0) = w(0) (1 - \mathbf{1}_{{\rm D}}),
    \end{cases}
\end{equation}
where $\mathbf{1}_{{\rm D}}$ is the characteristic function of the unit disk. 
In view of the streamlines of the velocity flow $u_{eq}$, see Figure \ref{fig:Lambdipole}, $w_{ext}$ remains spatially supported outside the unit disk. We define $
    w_{int}(t) := w(t) - w_{ext}(t) $,
and observe that it evolves according the forced linear equation 
\begin{equation} \label{Euler_moving_int}
     \partial_{t} w_{int}  = \mathcal{L}_{eq} \, w - \cL_{eq} w_{ext} = \mathcal{L}_{eq} \, w_{int} - u[w_{ext}] \cdot \nabla \omega_{eq}.
\end{equation}
We claim that ${\rm{supp}}(w_{int}(t)) \subseteq {\rm D}$ for every $t \geq 0$. Indeed, let $r(t):= w_{int}(t) - \mathbf{1}_{{\rm D}} w_{int}(t)$, where $\mathbf{1}_{{\rm D}}$ is the indicator function of the unit disk. We have
\begin{equation*}
    \partial_{t} r(t) = \mathcal{L}_{eq} \, w_{int} - \mathbf{1}_{{\rm D}} \,  \mathcal{L}_{eq} \, w_{int} = - u_{eq} \cdot \nabla r(t)
\end{equation*}
where the last identity descends from the fact that $u_{eq} \cdot \nabla \mathbf{1}_{{\rm D}}=0$ since $\nabla \mathbf{1}_{{\rm D}}$ is a distribution supported on $\partial {\rm D}$ and directed along its normal vector $\mathbf{n}$, whereas $u_{eq}$ is tangential to $\partial {\rm D}$ (the boundary of the disk is a streamline for the flow). By uniqueness of the solution for the transport equation arising from $r(0)=0$, we have that $r(t) $ vanishes identically. This proves our claim.

We observe that, as a consequence of the divergence Theorem, the total circulation $\int_{\R^2} w(t) {\rm d}x$ is constant in time. Moreover, since $w_{ext}$ is governed by the transport equation, its circulation is preserved at any time. We conclude that 
$$
\int_{{\rm D}} w_{int}(t,x) \, {\rm d}x = \int_{\R^2} w(t,x) \, {\rm d}x - \int_{\R^2 \setminus {\rm D}} w_{ext}(t,x) \, {\rm d}x =  \int_{{\rm D}} w(0,x) \, {\rm d}x
$$
is constant as well. 
{Let us consider $w_{int}(t) = \widetilde w(t) + \mu \ \mathbf{1}_{{\rm D}}$, where $\mu := |{\rm D}|^{-1} \int_{{\rm D}} w_{int}(0) \, {\rm d}x
$ and
$\widetilde w$ has zero average. By \eqref{Euler_moving_int},
\begin{equation}\label{complete-evolution}
\begin{aligned}
    \partial_{t} \widetilde w 
    &= - \diver(u_{eq} \widetilde w) - \diver(u[\widetilde w] \omega_{eq}) - u[\mu \ {\bf 1}_{\rm D}] \cdot \nabla \omega_{eq} - u[w_{ext}] \cdot \nabla \omega_{eq} \\
    &= \cL_{eq} \widetilde w - \frac{\mu}{2} f_1^{\rm ev} - u[w_{ext}] \cdot \nabla \omega_{eq} ,
\end{aligned}
\end{equation}
with $f_1^{\rm ev}$ in \eqref{f1oddev}. Indeed, by solving the Poisson problem $    \Delta \psi = \mu \ \mathbf{1}_{{\rm D}} \ {\rm in} \ \mathbb{R}^2$ with $ |\nabla \psi|$ goes to $0$ as $|x|\to +\infty$,
we find that the general solution inside the unit disk is $\psi = \frac{\mu}{4} |x|^2 + C$. Thus, $u[\mu] = \nabla^{\perp} \psi = \frac{\mu}{2} x^{\perp}$ on ${\rm D}$. Moreover, since $x^{\perp} \cdot \nabla=\partial_{\theta}$ in polar coordinates, we have $
  u[\mu] \cdot \nabla \omega_{eq} = \frac{\mu}{2} \partial_{\theta} \omega_{eq} = \frac{\mu}{2} f_1^{\rm ev}$ by \eqref{f1oddev}.
}

\subsection{The homogeneous problem} \label{passingtosL} We now restrict our analysis to the homogeneous part of \eqref{complete-evolution}. The linear evolution of an initial data supported on ${\rm D}$ with \emph{zero average} is determined by
\begin{equation}\label{questaqua}
    \partial_{t} w =  - u_{eq} \cdot \nabla w - \nabla^{\perp} \omega_{eq} \cdot u^{\perp}[w] =  - u_{eq} \cdot \nabla w +k^2 u_{eq} \cdot u^{\perp}[w],  
\end{equation}
since, by the functional relation \eqref{functionalrelation}, we have
\begin{equation}
    \nabla^\perp \omega_{eq} =-k^2 u_{eq}.
\end{equation}
By the incompressibility condition, we know that $u^{\perp}[w]$ is minus the gradient of a scalar function $\psi$ solving
\begin{equation} \label{Poisson_problem}
    \begin{cases}
        \Delta \psi =  w \qquad {\rm{in}} \ {\rm D}, \\
     \Delta \psi = 0 \qquad {\rm{in}} \ {\R^2 \setminus {\rm D}},
    \end{cases}
\end{equation}
with $|\nabla \psi|(x) \to 0$ as $x \to \infty$. The function $\psi_{ext} := \psi|_{\R^2 \setminus {\rm D}}$ is harmonic and with gradient decaying at infinity. Thus, in polar coordinates,
\begin{equation}\label{psiext}
\psi_{ext}(r,\theta) =  \zeta \ln(r) + \psi_0 +  \sum_{m \in \Z \setminus \{0\} } \psi_m r^{-|m|} e^{\im m \theta} \, ,
\end{equation}
for some complex coefficients $\zeta$ and $\psi_m $, $m\in \Z$. Since $w$ is average free, one has
\begin{equation}
    0 = \int_{\rm D} w(x) \, {\rm d}x = \int_{\rm D} \Delta \psi(x) \, {\rm d}x = \oint_{\partial {\rm D}} \partial_r \psi(\sigma) \, {\rm d}\sigma = 2\pi \zeta,
\end{equation}
where the last identity follows by differentiating \eqref{psiext} on
$\partial {\rm D}$ and integrating in $\theta$, since all the
non-zero Fourier modes have zero average. Hence $\zeta=0$.

As a consequence, at the boundary of the unit disk, $\psi_{ext}$ fulfills the relation
$$
(\partial_r \psi_{ext})(1,\theta)  = - |D_\theta| \psi_{ext}(1,\theta),
$$
where $|D_\theta|$ is the Fourier multiplier  
\begin{equation} \label{D-to-N}
    |D_\theta| (e^{i m \theta}) : = |m| e^{im\theta}.
\end{equation}
Thus, \eqref{Poisson_problem} is reduced to the following Poisson equation with Robin boundary condition
\begin{equation}\label{PoissonRobin}
    \begin{cases}
        \Delta \psi = w \qquad &{\rm{in}} \ {\rm D}, \\
        \partial_{n} \psi =- |D_\theta|\psi \qquad &{\rm{on}} \ \partial {\rm D}.
    \end{cases}
\end{equation}
Let $\Delta_{D\! N}^{-1}$ denote the operator that associates with $\omega \in L^2_0({\rm D})$ the unique (up to a constant) solution of \eqref{PoissonRobin}.  By the previous passages we have
\begin{equation}
    u^{\perp}[w] = -\nabla  \Delta_{D\! N}^{-1}w \, .
\end{equation}
Finally,
we write \eqref{questaqua} as
\begin{align}\label{effectivesystem}
    \partial_{t} w = - u_{eq} \cdot \nabla w - k^2 u_{eq}  \cdot \nabla (\Delta_{D\! N}^{-1} \, w ) = - (u_{eq} \cdot \nabla) ({\rm Id} + k^2 \Delta_{D\! N}^{-1}) w = \sL w,
\end{align}
where $\sL$ is the operator introduced in \eqref{Linear_operator}.  

We end this section by endowing $L^2_0({\rm D})$ with an explicit basis given by  eigenfunctions of the self-adjoint and compact operator $\Delta_{D\!N}^{-1}$.
Let $m\geq 0$ be an integer. We recall that the $m$-th Bessel function of the first kind $J_m:\R \to \R$ is defined as the solution of the following ODE
\begin{equation}
    \label{def:Jm}
  t^2 J_m''(t) + t J_m'(t) + (t^2 - m^2 ) J_m(t) = 0,
\end{equation}
so that, for every $\alpha \in \R$,
$
    \Delta [ J_m (\alpha r) e^{\pm \im m \theta} ] = -\alpha^2 J_m(\alpha r)  e^{\pm \im m \theta}  
$. Let $\alpha \neq 0$,
by \eqref{PoissonRobin}, $w:= J_m(\alpha r)  e^{\pm \im m \theta}$ is such that $\Delta_{D\!N}^{-1}w = -\alpha^{-2} w  $ if and only if $\psi := -\alpha^{-2}w $ fulfills $\partial_n \psi = - |D_\theta| \psi$. The latter condition boils down to
$
\alpha J'_m(\alpha) + m J_m(\alpha) = 0 
$.
By recalling the well-known relation $xJ_m'(x) = x J_{m-1}(x) - m J_m(x) $, we conclude that $f$ is an eigenfunction of $\Delta_{D\!N}^{-1}$ if
\begin{equation}
    \alpha J_{m-1}(\alpha)  = 0 \, , 
\end{equation}
i.e\  $\alpha\in \{ \pm j_{m-1,n} \}_{n\geq 1} $ where $j_{m,n}$ is the $n$-th strictly positive root of $J_{m}(x)$. 
One has $w \in L^2_0({\rm D})$ since
\begin{align*}
    \int_{\rm D} w(x) \, {\rm d}x = \int_{\rm D} \Delta \Delta_{D\!N}^{-1} w(x) \, {\rm d}x = \oint_{\partial {\rm D}} \partial_n \Delta_{D\!N}^{-1} w(x) \, {\rm d}\sigma =  \oint_{\partial {\rm D}} \partial_n \psi(\sigma) \, {\rm d}\sigma \stackrel{\eqref{PoissonRobin}}{=} -\oint_{\partial {\rm D}} |D_\theta| \psi(\sigma) \, {\rm d}\sigma  \stackrel{\eqref{D-to-N}}{=} 0 \, . 
\end{align*} 
Let us observe that each eigenfunction associated with a negative value of $\alpha$ is a multiple of another one associated with $|\alpha|$ because of the parity properties of the Bessel functions.
In conclusion,
thanks to the Hilbert-Schmidt theorem,  the following set forms a basis of $L^2_0({\rm D})$
\begin{equation} \label{Besselbasis}
    \mathcal{B}:= \{J_0(j_{1,n} r)\}_{n\geq 1} \cup \{ J_{m}(j_{m-1,n} r) e^{\im m \theta} \}_{m\geq 1, n\geq 1}.
\end{equation}

We are now in a position to study the linear system \eqref{effectivesystem} generated by the operator $\sL$ in \eqref{Linear_operator}.

\subsection{The Hamiltonian operator $\sL$}\label{sec:Hamiltonian} In this section we conclude our proof of Theorem \ref{Spectralstability_Thm}, by a complete analysis of the spectral properties of the operator $\sL$ governing \eqref{effectivesystem}.

The operator $\sL$ in \eqref{Linear_operator} is {\it Hamiltonian} in the sense of \cite[Definition 1.4]{InstabilityCellularFlow}, namely it decomposes into
\begin{equation}\label{Hamiltonian}
    {\sL} = {\mathcal{J}} {\mathcal{H}}, \qquad {\rm{with}} \qquad \mathcal{J} := - u_{eq} \cdot \nabla, \qquad {\rm{and}} \qquad \mathcal{H} := {\rm Id} + k^2 \Delta_{D\! N}^{-1},
\end{equation}
where:
\begin{enumerate}[label=(\roman*)]
    \item the operator $\mathcal{J}: {\rm Dom}(\mathcal{J}) \subset L_{0}^{2}({\rm D}) \rightarrow L_{0}^{2}({\rm D})$, with ${\rm Dom}(\mathcal{J}):= \left\{ f \in L_{0}^{2}({\rm D}) : \mathcal{J}f \in L_{0}^{2}({\rm D}) \right\}$, is densely defined, closed and skew-adjoint;
    \item the operator $\mathcal{H}: L_{0}^{2}({\rm D}) \rightarrow L_{0}^
    {2}({\rm D})$ is bounded and self-adjoint.
\end{enumerate}
In particular, the operator $\sL$ is closed with ${\rm Dom}(\sL) := {\rm Dom}(\mathcal{J})$.

We are interested in all possible initial data $w_0 \in L^2_0({\rm D})$ such that the evolution $e^{t\sL} w_0 $ grows in time.
A first answer to the problem is provided by classifying
every \emph{unstable generalized eigenfunction} of  $\sL$. We recall that $f \in \cap_{m=1}^\infty {\rm Dom}(\sL^m)$ is a generalized eigenfunction associated with an eigenvalue $\lambda \in \C$ if, for some integer $m \geq 1$, one has 
\begin{equation}\label{def:geneig}
    (\sL - \lambda)^m f = 0\, .
\end{equation}
We shall call such an $f$ \emph{unstable} if (cfr.\ \cite[p.\ 13]{InstabilityCellularFlow})
$$
\text{\it either}\quad ({\bf A})\ \ \Re\, \lambda \neq 0\quad \text{\it or}\quad ({\bf B})\ \ \Re\, \lambda = 0 \text{ and } m\geq 2\,.
$$
Indeed, one can readily prove that $\|e^{t\sL}f\|_{L^2} $ will grow exponentially fast in the first case, and polynomially fast in the second case. 

Suppose $m\geq 1$ is the first integer for which \eqref{def:geneig} holds. If $m=1$ then $f_0 := f \neq 0$ is an actual eigenvector of $\sL$. On the other hand if, $m\geq 2$, we can associate with $f$ the eigenvector $f_0 := (\sL - \lambda)^{m-1} f\neq 0$. In both cases, we have
\begin{equation}
    \label{Kreinsignature}
  f_0 \in {\rm Ran}(\sL)\quad\text{and}\quad  \langle \cH f_0 , f_0 \rangle = 0,
\end{equation}
as deduced, e.g., in the first lines of the {\it Alternative proof} of \cite[Proposition 2.9]{InstabilityCellularFlow}.

The self-adjoint operator $\mathcal{H} = {\rm Id} + k^2 \Delta_{D\! N}^{-1}$ is diagonal in the basis $\mathcal{B}$ in \eqref{Besselbasis}, since
\begin{equation}\label{actionH}
   \mathcal{H}(J_{0}( j_{1,n} r)) = \left( 1 - \frac{k^2}{j_{1,n}^{2}} \right) J_{0}(j_{1,n} r), \quad \mathcal{H} (J_{m}( j_{m-1,n} r) e^{ \im m \theta})  = \left( 1 - \frac{k^2}{j_{m-1,n}^{2}} \right) J_{m}(j_{m-1,n} r) e^{\im m \theta}.
\end{equation}
The kernel ${\rm K}$ of $\cH$ and its maximal negative invariant subspace ${\rm N}$ are respectively given by 
\begin{equation}\label{negativespace}
    {\rm{K}} = {\rm{span}}\{ J_{0}(k r), \ J_2(k r) \cos(2\theta), \  J_2(k r) \sin(2\theta)\},\quad {\rm{N}} = {\rm{span}}\{ J_1( j_{0,1} r) \cos(\theta), \ J_1( j_{0,1} r) \sin(\theta) \}.
\end{equation}
In particular the functions $f_0^{\rm odd}$ and $f_0^{\rm ev}$ in \eqref{f0oddev} lie in the kernel ${\rm K}$.
To prove \eqref{negativespace} we observe that, in view of \eqref{actionH}, any nonpositive direction of $\cH$ corresponds to an integer $m \geq 1$ such that
$$
j_{m-1,n} \leq k = j_{1,1}.
$$
By definition 
$$
    j_{m-1,1} < j_{m-1,2} < j_{m-1,3} < \dots 
$$
and, by properties of  the Bessel functions' roots, for every integer $n\geq 1$ one has
\begin{equation}\label{secondpropBessel}
        j_{0,n} < j_{1,n} < j_{2,n} < j_{3,n} <  \dots 
    \end{equation}
Thus, to conclude \eqref{negativespace}, it suffices to prove that $j_{1,1} < j_{0,2}$. The latter inequality descends from the identity $J_0'(x) = -J_1(x)$, whence one deduces that the function $J_0$, which is $1$ at $x=0$, is monotonically decreasing in the interval $(0,j_{1,1})$ where $j_{0,1}$ lies (by \eqref{secondpropBessel}). By Rolle's theorem there will be a zero of $J_1$ in the interval $(j_{0,1},j_{0,2})$ and by the considerations above this is exactly $j_{1,1}$. 
\begin{figure}[h!]
    \centering
    {\footnotesize
    \begin{tikzpicture}[scale=0.5, transform shape]
      \matrix (m) [matrix of math nodes, 
                   row sep=1em,      
                   column sep=1em,   
                   nodes={anchor=center}] 
      {
        J_{2,1} & < & J_{2,2} & < & \dots &   &       \\
        \vee    &   & \vee    &   & \vee    &   &       \\
        |[draw=red, thick, inner sep=2pt, name=j11]| J_{1,1} & < & J_{1,2} & < & \dots &  \\
        \vee    &   &    \vee     &   & \vee    &   &       \\
        J_{0,1} & < & 
        |[draw=red, thick, inner sep=1.5pt, name=j02]| J_{0,2} & < & \dots &   &       \\
      };
      
      \draw[red, thick] (j11.south east) -- (j02.north west) node[midway, sloped, above, text=red, font=\Large] {$<$};
    \end{tikzpicture}
    $\qquad\qquad$
    \begin{tikzpicture}[x=1cm,y=2.3cm]
    \draw[->] (-0.1,0) -- (5.9,0) node[right] {$x$};
    \draw[->] (0,-0.45) -- (0,1.08);

    \draw (0,0.025) -- (0,-0.025);
    \node at (-0.18,-0.10) {$0$};

    \foreach \x/\lab in {
        2.4048/{j_{0,1}},
        3.8317/{j_{1,1}},
        5.5201/{j_{0,2}}
    }{
        \draw (\x,0.025) -- (\x,-0.025)
            node[below=4pt] {$\lab$};
    }

    \foreach \x in {2.4048,3.8317,5.5201}{
        \draw[dashed,gray] (\x,-0.45) -- (\x,1.05);
    }

    \draw[thick,blue]
    plot[smooth] coordinates {
    (0.000,1.000) (0.250,0.984) (0.500,0.938) (0.750,0.864) (1.000,0.765)
    (1.250,0.646) (1.500,0.512) (1.750,0.370) (2.000,0.224) (2.250,0.083)
    (2.405,0.000) (2.750,-0.164) (3.000,-0.260) (3.250,-0.332) (3.500,-0.380)
    (3.750,-0.402) (4.000,-0.397) (4.250,-0.366) (4.500,-0.320) (4.750,-0.265)
    (5.000,-0.178) (5.250,-0.091) (5.520,0.000) (5.800,0.091)
    };

    \draw[thick,red]
    plot[smooth] coordinates {
    (0.000,0.000) (0.250,0.124) (0.500,0.242) (0.750,0.349) (1.000,0.440)
    (1.250,0.511) (1.500,0.558) (1.750,0.580) (2.000,0.577) (2.250,0.548)
    (2.500,0.497) (2.750,0.425) (3.000,0.339) (3.250,0.242) (3.500,0.137)
    (3.832,0.000) (4.000,-0.066) (4.250,-0.154) (4.500,-0.231) (4.750,-0.292)
    (5.000,-0.328) (5.250,-0.341) (5.520,-0.340) (5.800,-0.309)
    };

    \node[blue] at (1.0,0.86) {$J_0$};
    \node[red] at (1.7,0.68) {$J_1$};
    \end{tikzpicture}
    }
    \caption{Graphic representation of two key elements in the proof of \eqref{negativespace}. On the left, we represent the ordering of Bessel functions' roots. On the right, a graphic visualization of the inequality $j_{1,1} < j_{0,2}$.}
    \label{fig:schemino}
\end{figure}
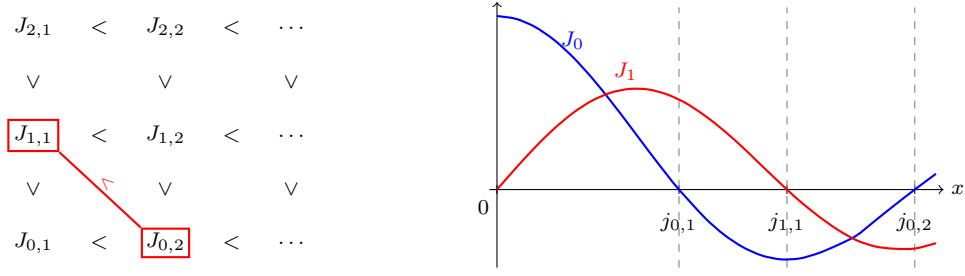
\begin{rmk}
In the sequel, we will exclude unstable eigenvalues $\lambda \in \sigma(\sL) $, $\Re(\lambda)>0$, from the spectrum of $\sL$. 
    By the exponential trichotomy of Lin and Zeng \cite[Theorem 2.2]{LinZeng2022},  one concludes that there exists $C>0$ such that the semigroup $e^{t \sL}$ generated by $\sL$ is bounded by
    $$
    \| e^{t \sL} w \|_{L^2_0({\rm D})} \leq C (1+t^5) \| w \|_{L^2_0({\rm D})}, \quad \text{for every }t>0\text{ and }w\in L^2_0({\rm D}),
    $$
    where the exponent $5 = 1 + 2 \dim {\rm N}$ descends from \cite[formula (2.8)]{LinZeng2022}. The following fine-tuned analysis will be able to show the more precise estimate \eqref{finalestimate}, from which we will then deduce estimate \eqref{quadraticgrowth} on the semigroup $e^{t \cL_{eq}}$.
\end{rmk}
By general properties of Hamiltonian operators, see e.g.\ \cite[Lemma 2.1]{InstabilityCellularFlow}, formula \eqref{negativespace} entails that the operator $\sL$ has at most two linearly independent unstable eigenvectors. 
In our case, however, the negative directions of $\cH$ are not related to some unstable direction of $\sL$, but are absorbed by the Jordan chains generated by the symmetries listed in Section \ref{symmetries} that we described in Lemma \ref{generalizedkernel}. {This property is presented in the following result.}

\begin{lem}\label{stability_condition}
The functions $f_1^{\rm odd}$ and $f_1^{\rm ev}$ in \eqref{f1oddev} satisfy
   $ \langle \cH f_{1}^{\rm odd}, f_{1}^{\rm odd} \rangle_{L^{2}({\rm D})} = \langle \cH f_{1}^{\rm ev}, f_{1}^{\rm ev} \rangle_{L^{2}({\rm D})} < 0$.
\end{lem}
\begin{proof} We compute $\frac{J_0(k)}{2k} \Delta_{D\!N}^{-1} f_1^{\rm odd} = g(k r) \sin(\theta)$ where
$$
t^2 g''(t) + t g'(t) -g(t) = \frac{t^2}{k^2} J_1(t)\ \text{ for } t \in \R \text{ and }\
k g'(k) = - g(k).
$$
One verifies that $g(t) = k^{-2} (\tfrac12 J_1'(k) t - J_1(t)  )$. 
It follows that
$
    \frac{J_0(k)}{2k} \cH(f_{1}^{\rm odd}) = \frac{k}{2} J_{1}^{'}(k) r \sin (\theta). 
$
Then
\begin{equation*}
\begin{aligned}
    \frac{J_0(k)^2}{4k^2}  \langle \cH f_{1}^{\rm odd}, f_{1}^{\rm odd} \rangle_{L^{2}({\rm D})} &= \int_0^{2\pi} \int_0^1 \frac{
        k}{2} r^2 J_1'(k) J_1(kr) \sin^{2}(\theta)  \, {\rm d}r \, {\rm d}\theta \\
&= \frac{\pi k}{2} J_1'(k) \int_0^1 r^2 J_1(kr) \, {\rm d}r =  \frac{\pi }{2k^2} J_1'(k) \int_0^k z^2 J_1(z) \, {\rm d}z ,
\end{aligned}
    \end{equation*}
    where in the last step we performed the change of variable $z=kr$. By the recurrence formulas for Bessel functions in \cite[Appendix~B, Section~B.2]{grafakos2008}, we have the following identities
\begin{equation*}
    \int_{0}^{k} z^2 J_{1}(z) \, {\rm d}z = k^2 J_2(k), \qquad J_1'(k) = J_0(k), \qquad
J_2(k) = - J_0(k).
\end{equation*}
We conclude that $
  \langle \cH f_{1}^{\rm odd}, f_{1}^{\rm odd} \rangle = -2 \pi k^2  < 0$. The computation for $\langle \cH f_{1}^{\rm ev}, f_{1}^{\rm ev} \rangle$ is identical.
\end{proof}

The above lemma justifies the following  characterization.

\begin{cor}\label{Jordanchainlength2} 
    The Jordan chains $\{f_1^{\rm odd},f_0^{\rm odd}\} $ and $\{f_1^{\rm ev},f_0^{\rm ev}\} $ of the operator $\sL$ introduced in Lemma \ref{generalizedkernel} do not admit any extension, in the sense that there does not exist $f_2\in L^2_0({\rm D})$ such that $ \sL f_2 = f_1^{\rm odd} $ or $ \sL f_2 = f_1^{\rm ev} $.
\end{cor}
\begin{proof}
Indeed one would have $0 \neq \langle \cH  f_1^{\rm odd},  f_1^{\rm odd} \rangle =  \langle \cH  f_1^{\rm odd},  \cJ \cH f_2 \rangle = - \langle f_0^{\rm odd},  \cH f_2 \rangle  = 0$, which is  a contradiction. The same holds in the other case.
\end{proof}
\begin{rmk}
    \label{rmk:dualJordanchain}
Canonically, one can construct Jordan chains for $\sL^*=-\cH \cJ$ as $\{ \cH f_1^{\rm odd}, \cH f_0^{\rm odd} \}$ and $\{ \cH f_1^{\rm ev}, \cH f_0^{\rm ev} \}$. However, since $f_0^{\rm odd}$ and $f_0^{\rm ev}$ belong to the kernel of $\cH$, this construction just provides two kernel elements $g_0^{\rm odd}:= \cH f_1^{\rm odd}$, $ g_0^{\rm ev}:= \cH f_1^{\rm ev}$ of $\sL^*$. Only the last one gives rise to a nontrivial Jordan chain. Indeed, setting $g_1^{\rm ev}(r) := \tfrac12r^2 k^2$ one has $\sL^*  g_1^{\rm ev} = \cH f_{1}^{\rm{ev}}$ by inspection. On the contrary, since $f_1^{\rm odd}$ lies in ${\rm Ker}(\cJ)$, one readily obtains that there does not exist any $g_1^{\rm odd} \in L^2_0({\rm D})$ such that $ \sL^* g_1^{\rm odd} = \cH f_1^{\rm odd}$. 

As anticipated in the introduction, such a vector $g_1^{\rm odd}$ would have allowed us to define the subspace $Y := \{g_0^{\rm odd},g_1^{\rm odd},g_0^{\rm ev},g_1^{\rm ev}\}^{\perp} $  being closed, $\sL$-invariant and supporting $\langle\cH\cdot,\cdot\rangle$ coercively. The argument below circumvents this obstruction by using a different invariant structure.
\end{rmk}
In light of Lemma \ref{stability_condition},
we define the auxiliary closed
subspaces $\tilde Z \subset Z \subset L^2_0({\rm D})$ given by
    \begin{equation}\label{def:Z}
    \begin{aligned}
        &Z := \{ f \in L^2_0({\rm D}) : \langle \cH f, f_{1}^{\rm odd} \rangle = \langle \cH f, f_{1}^{\rm ev} \rangle = 0 \},\\  
        &\tilde Z := \{f \in Z : \langle \cH f, h \rangle = 0\text{ for every }h \in {\rm Ker}(\sL) \}\, .
    \end{aligned}
    \end{equation}
Let us list their main properties.
\begin{enumerate}[label=(\roman*)]
    \item Since $\langle \cH f_{1}^{odd}, f_1^{odd} \rangle  = \langle \cH f_{1}^{ev}, f_1^{ev} \rangle\neq 0 $ and $\langle \cH f_{1}^{\rm odd}, f_1^{\rm ev} \rangle = 0$, one readily obtains that
$$
 L^2_0({\rm D}) =  {\rm span} \{ f_1^{\rm odd}, f_1^{\rm ev} \} \oplus Z .
$$ 
\item \label{eigsinZ} One has ${\rm Ran}(\sL) \subseteq \tilde Z\subset  Z$. Indeed, if $f= \cJ \cH g$ then 
$$
\langle \cH  f, f_{1}^{\rm odd} \rangle = -\langle \cH g, \sL f_{1}^{\rm odd}  \rangle = - \langle \cH g, f_{0}^{\rm odd}\rangle = - \langle  g, \cH f_{0}^{\rm odd} \rangle = 0, 
$$
and similarly $\langle \cH  f, f_{1}^{\rm ev} \rangle=0$. Moreover, for every $h $ in $ {\rm Ker}(\sL) $, $\langle \cH f , h \rangle = - \langle \cH g ,\sL h \rangle = 0 $.
\item \label{positivesemidefinite} The operator $\mathcal{H}$ is positive semi-definite on $Z$. Indeed, any $f \in Z$ such that $\langle \mathcal{H} f, f \rangle < 0$ would be a third negative direction for $\cH$ linearly independent from $f_1^{\rm odd}$ and $f_1^{\rm ev}$, in contradiction with \eqref{negativespace}.
\end{enumerate}

 A direct consequence of property \ref{positivesemidefinite} of $Z$ is the following
\begin{lem}\label{lemZ}
    Let $f \in Z$ be such that $\langle \cH f ,f \rangle = 0$, then $\cH f = 0$.
\end{lem}
\begin{proof}
The quadratic form $q: L^2_0({\rm D}) \to \R$ given by $q(f):=\langle \mathcal{H} f, f \rangle$, which is positive semi-definite on $Z$, attains a minimum at  $f$ when constrained on the set $\{g\in Z : \| g \|_{L^2} = \| f \|_{L^2} \} $.
By the Lagrange multiplier method, $f$ will be a stationary point of the following functional
\begin{equation*}
    J(g, \mu_0, \mu_1, \mu_2) := \frac{1}{2} \langle \mathcal{H} g, g \rangle - \frac{1}{2} \mu_0 (\| g \|^2 - \|f\|^2) - \mu_1 \langle \mathcal{H}  g, f_{1}^{\rm odd} \rangle - \mu_2 \langle \mathcal{H}  g, f_{1}^{\rm ev} \rangle,
\end{equation*}
whose gradient with respect to $g$ at $g=f$ gives
\begin{equation*}
        \mathcal{H}  f - \mu_0 f - \mu_1 \mathcal{H}  f_{1}^{\rm odd} - \mu_2 \mathcal{H} f_{1}^{\rm ev} = 0.
\end{equation*}
Since, by hypothesis, $\langle \mathcal{H} f, f \rangle = 0$ and $f\in Z$, the inner products of the left-hand side above with the functions $f$, $f_1^{\rm odd}$ and $f_1^{\rm ev}$ respectively give $\mu_0 = \mu_1 = \mu_2 = 0$. Thus $
\cH f = 0$.
\end{proof}
We also obtain the following
\begin{cor}
    The operator $\cH$ is coercive on $Z\cap {\rm K}^\perp$, namely there exists $c>0$ such that
    \begin{equation}
        \label{coercivityH}
        \langle \cH f,f\rangle \geq c \|f\|^2,\quad \text{for every }f\in Z\cap {\rm K}^\perp.
    \end{equation}
\end{cor}
\begin{proof}
    By contradiction, let $\{f_n\}_{n\in \N} \subset Z\cap {\rm K}^\perp $, with $\|f_n\|=1$ for every $n\in \N$, be such that $ \langle \cH f_n,f_n\rangle \to 0$ as $n\to +\infty$. Up to subsequences, $f_n$ converges to some $f \in Z\cap {\rm K}^\perp$ as $n\to +\infty$ weakly in $L^2_0({\rm D})$. By \eqref{Hamiltonian},
    \begin{equation}
        \label{contradictionsource}
    0 \stackrel{n\to +\infty}{\longleftarrow} \langle \cH f_n,f_n\rangle = 1  + k^2 \langle \Delta_{D\! N}^{-1} f_n, f_n \rangle \stackrel{n\to +\infty}{\longrightarrow} 1 +  \langle \Delta_{D\! N}^{-1} f, f \rangle \geq \|f\|^2 +  \langle \Delta_{D\! N}^{-1} f, f \rangle = \langle \cH f,f\rangle,
    \end{equation}
    where we used that $\Delta_{D\! N}^{-1} f_n$ converges strongly to $\Delta_{D\! N}^{-1} f$ as $n\to +\infty$ because $\Delta_{D\! N}^{-1}$ is a compact operator. It follows that $\langle H f, f\rangle = 0$ and, by Lemma \ref{lemZ}, $f \in {\rm K}$. Thus $f=0$, and by \eqref{contradictionsource}, $\langle \cH f_n,f_n\rangle \to 1$ as $n\to +\infty$, which is a contradiction.
\end{proof}
We now prove a technical result  that will be fundamental in the sequel.
\begin{lem}\label{toSpain}
    One has $\overline{{\rm Ran}(\cJ)} \cap {\rm K} = {\rm span}\, \{f_0^{\rm odd},f_0^{\rm ev} \} $.
\end{lem}
\begin{proof}
By Lemma \ref{generalizedkernel}, $\overline{{\rm Ran}(\cJ)} \cap {\rm K} \supseteq {\rm span}\, \{f_0^{\rm odd},f_0^{\rm ev} \} $ where, by \eqref{negativespace}, ${\rm K}$ is a three-dimensional subspace of $L^2_0({\rm D})$. Then the proof boils down to showing that $f:=J_0(kr)$ does not belong to $\overline{{\rm Ran}(\cJ)}$. Let us introduce the orthogonal projection $P$ onto ${\rm Ker}(\cJ)$, which is closed because $\cJ$ is a closed operator. Clearly, if $Pf \not \equiv 0$ then $f \notin \overline{{\rm Ran}(\cJ)}= {\rm Ker}(\cJ)^\perp$. $P$ is characterized as the operator that averages its input on the streamlines of the Lamb-Chaplygin dipole in \eqref{Lambdipole}, see \cite[Lemma 2.3]{Lin2004}. In particular, at the elliptic point $(r_*,\theta_*) = (z_*/k, \pi/2)$, where $z_*$ is the first positive zero of $J_1'$, the function $Pf$ attains the value of $f$, namely $J_0(z_*)$. By exploiting the identity $J_1'(z)=J_0(z)-J_1(z)/z$, we obtain that $J_0(z_*)= J_1(z_*)/z_*$ which is different from zero because $J_1>0$ on the interval $(0,k)$, see Figure \ref{treBessel}. Thus $P(J_0(kr))$ is not identically vanishing.
\end{proof}
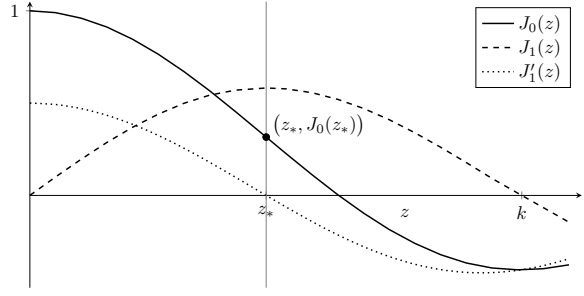
\begin{figure}[h!!!]
\centering
\begin{tikzpicture}[scale=0.7, transform shape]
\begin{axis}[
    width=12cm,
    height=7cm,
    xmin=0, xmax=4.3,
    ymin=-0.5, ymax=1.05,
    axis lines=middle,
    xlabel={$z$},
    ylabel={},
    xtick={1.841,3.832},
    xticklabels={$z_*$,$k$},
    ytick={1},
    legend style={at={(0.98,0.98)},anchor=north east},
    samples=100,
    domain=0:4.3,
    clip=false,
]
\addplot[thick] coordinates {
(0,1.000) (0.2,0.990) (0.4,0.960) (0.6,0.912) (0.8,0.846)
(1.0,0.765) (1.2,0.671) (1.4,0.567) (1.6,0.455) (1.841,0.316)
(2.0,0.224) (2.2,0.111) (2.4,0.003) (2.6,-0.096) (2.8,-0.185)
(3.0,-0.260) (3.2,-0.321) (3.4,-0.366) (3.6,-0.393) (3.8,-0.403)
(4.0,-0.397) (4.2,-0.376)
};
\addlegendentry{$J_0(z)$}

\addplot[thick,dashed] coordinates {
(0,0.000) (0.2,0.100) (0.4,0.196) (0.6,0.287) (0.8,0.369)
(1.0,0.440) (1.2,0.499) (1.4,0.543) (1.6,0.570) (1.841,0.582)
(2.0,0.577) (2.2,0.556) (2.4,0.520) (2.6,0.471) (2.8,0.410)
(3.0,0.339) (3.2,0.261) (3.4,0.178) (3.6,0.092) (3.832,0.000)
(4.0,-0.066) (4.2,-0.145)
};
\addlegendentry{$J_1(z)$}

\addplot[thick,dotted,smooth] coordinates {
(0.00000,0.500000)
(0.20000,0.492521)
(0.40000,0.470332)
(0.60000,0.434170)
(0.80000,0.385235)
(1.00000,0.325147)
(1.20000,0.255892)
(1.40000,0.179750)
(1.60000,0.099217)
(1.75000,0.037515)
(1.84118,0.000002)
(1.90000,-0.024054)
(2.00000,-0.064472)
(2.20000,-0.142348)
(2.40000,-0.214236)
(2.60000,-0.277889)
(2.80000,-0.331361)
(3.00000,-0.373072)
(3.20000,-0.401858)
(3.40000,-0.417009)
(3.60000,-0.418287)
(3.80000,-0.405930)
(4.00000,-0.380639)
(4.20000,-0.343546)
};
\addlegendentry{$J_1'(z)$}
\addplot[gray] coordinates {(1.841,-0.5) (1.841,1.05)};
\addplot[only marks,mark=*,mark size=1.8pt] coordinates {(1.841,0.316)};
\node at (axis cs:2.25,0.38) {$\bigl(z_*,J_0(z_*)\bigr)$};
\end{axis}
\end{tikzpicture}
\caption{\label{treBessel} Qualitative plot of $J_0$, $J_1$, and $J_1'$. Here $z_*$ denotes the first positive zero of $J_1'$, while $k=j_{1,1}$ is the first positive zero of $J_1$.}
\end{figure}

We are now in a position to prove Theorem \ref{Spectralstability_Thm}. 
\begin{proof}[Proof of Theorem \ref{Spectralstability_Thm}] The operator $\sL$ in \eqref{Linear_operator} is a relatively compact perturbation of the skew-adjoint operator $\cJ$. Thus the set $\sigma_{L^2_0({\rm D})}(\sL) \setminus \im \R$ is contained in the pure-point spectrum of $\sL$. Regarding the latter, let $f$ be an  unstable generalized eigenfunction of $\sL$ associated with an eigenvalue $\lambda \in \C$. As done in \eqref{Kreinsignature}, we can construct an eigenfunction $f_0\in {\rm Ran}(\sL)$ such that $\langle \cH f_0, f_0\rangle = 0 $. By property \ref{eigsinZ} of $Z$, we have $f_0\in Z$ and, by Lemma \ref{lemZ}, $\lambda = 0$. Thus
any unstable generalized eigenfunction of $\sL$ is associated with $0$. This concludes the proof of \eqref{primopunto}.

Let us consider the point \eqref{secondopunto}. By Corollary \ref{Jordanchainlength2}, we already have two Jordan chains of length $2$ in the generalized kernel of $\sL$. Let us suppose that another Jordan chain of arbitrary length, ending as $\{\dots,f_1,f_0\}$ with $\sL f_0= 0 $ and $\sL f_1 = f_0$, appears in the generalized kernel. Then, as for the proof of the first point, $f_0\in Z$ and $\langle \cH f_0, f_0\rangle =0 $, implying $f_0 \in {\rm Ran}(\cJ) \cap {\rm Ker}(\cH)$. By Lemma \ref{toSpain} we obtain that $f_0=\alpha f_0^{\rm odd} + \beta f_0^{\rm ev}$  for some $\alpha,\beta \in \C$. As a consequence $f_1 = \alpha f_1^{\rm odd} + \beta f_1^{\rm ev} + {\rm Ker}(\sL)$ and the Jordan chain we are considering is dynamically equivalent to a linear combination of the two in Corollary \ref{Jordanchainlength2}.
\end{proof}

\subsection{Linear stability} \label{sec:linearstability}
We now focus on the proof of Theorem \ref{lineardyn_Thm}. A first step is a weak separation result between the space ${\rm Ran}(\sL)\subseteq \tilde Z$ in \eqref{def:Z}, to which all trajectories are parallel, and the kernel ${\rm K}$ of the operator $\cH$, near which the Hamiltonian structure of the operator $\sL$ is insufficient to control the linear evolution.
\begin{lem} \label{banal_intersection}
    One has $\sL(\tilde Z)\cap {\rm K} =\{0\} $.
\end{lem}
\begin{proof}
    Let ${\rm K} \ni f_0= \sL f_1 $ for some $f_1 \in {\tilde Z}$. In particular, $f_0$ lies in $\overline{{\rm Ran}(\cJ)} \cap {\rm K} $ and, by Lemma \ref{toSpain}, there exist $\alpha,\beta\in \R$ such that
    $
    f_0 = \alpha f_0^{\rm odd} + \beta f_0^{\rm ev}
    $. In view of \eqref{actionJordan},
    \begin{equation}\label{decomposition}
   f_1 = \alpha  f_1^{\rm odd} - \beta f_1^{\rm ev} + h, \quad h\in {\rm Ker}(\sL).
    \end{equation}
    One readily obtains that $\beta=0$ in \eqref{decomposition}. Indeed, by \eqref{def:Z},
    \begin{equation}\label{cancelationf1ev}
    0 = \langle \cH f_1 , f_1^{\rm ev} \rangle =-\beta \langle  \cH f_1^{\rm ev} ,  f_1^{\rm ev} \rangle + \langle  \cH h ,  f_1^{\rm ev} \rangle.
    \end{equation}
    Let us observe that $ f_1^{\rm ev} = \cJ g$, with $g(r):=  - \tfrac12 {r^2 k^2}{\bf 1}_{{\rm D}} $, by inspection.
     As a consequence, 
     \begin{equation}\label{cancelationoccurs}
    \langle  \cH h ,  f_1^{\rm ev} \rangle = \langle  \cH h,  \cJ g \rangle = - \langle  \sL h ,   g \rangle = 0.
     \end{equation}
     Since, by Lemma \ref{stability_condition}, $\langle  \cH f_1^{\rm ev} ,  f_1^{\rm ev} \rangle \neq 0$, we conclude from \eqref{cancelationf1ev}-\eqref{cancelationoccurs} that  $\beta$ vanishes. We now prove that $\alpha=0$ too. Again, by \eqref{def:Z} and \eqref{decomposition},
    \begin{equation}\label{la1}
    \begin{cases}
    0 = \langle \cH f_1 , f_1^{\rm odd} \rangle =\alpha \langle  \cH f_1^{\rm odd} ,  f_1^{\rm odd} \rangle + \langle  \cH h ,  f_1^{\rm odd} \rangle, \\
    0 = \langle \cH f_1 , h \rangle =\alpha \langle  \cH f_1^{\rm odd} ,  h \rangle + \langle  \cH h , h \rangle.
    \end{cases}
    \end{equation}
    As a consequence
    $$
    \langle \cH f_1 , f_1 \rangle \stackrel{\eqref{decomposition}}{=} \alpha \big(\alpha \langle  \cH f_1^{\rm odd} ,  f_1^{\rm odd} \rangle + \langle  \cH h ,  f_1^{\rm odd} \rangle\big) + \big( \alpha \langle  \cH f_1^{\rm odd} ,  h \rangle + \langle  \cH h , h \rangle \big) \stackrel{\eqref{la1}}{=} 0.
    $$
   By Lemma \ref{lemZ} we obtain that $\cH f_1 =0$. Thus, $0=\cJ \cH f_1 =\alpha f_0^{\rm odd}$, namely $\alpha = 0$ and $f_0$ vanishes.
\end{proof}
As a corollary, let $\{g_n\}_{n\in \N}\subset \tilde Z$ be a bounded sequence  and suppose that $\sL g_n$ converges to some function $h \in {\rm K}$ weakly in $L^2_0({\rm D})$. Then, up to subsequences, $g_n$ tends to some $g_\star \in \tilde Z$ weakly in $L^2_0({\rm D})$. Since $\sL$ is a closed operator, $\sL g_\star = h$ and, by the previous lemma, $h = 0$.

Relying on the latter result we show that no initial datum in ${\rm Ran}(\sL)$ yields linear growth.

\begin{lem} \label{thanksGodforthislemma}
    For every $f \in {\rm Ran}(\sL)$ one has $
    \sup_{t \geq 0}     \| e^{t\sL} f \|_{L^2_0({\rm D})} < + \infty$.
\end{lem}
\begin{proof}
We argue by contradiction and suppose that there exists a sequence of times $t_n \geq 0$ such that  $\| e^{t_n\sL} f \|\to +\infty$ as $n \to +\infty$. 
Let $h_n :={e^{t_n \sL} f}/{\| e^{t_n \sL} f \|}$. Up to a subsequence, $h_n$ converges  to some $ h \in \Tilde{Z}$ weakly in $L^2_0({\rm D})$. We claim that the convergence is automatically strong and $h \in K$. Indeed, by the conservation of the energy $\cH$ along the trajectories of the semigroup, $\langle \cH e^{t_n \sL} f, e^{t_n \sL} f \rangle = \langle \cH f, f \rangle$, and, dividing by $\| e^{t_n \sL} f \|^2$, we get $\langle \cH h_n, h_n \rangle \to 0$ as $n \to +\infty$. On the other hand, since $ \Delta_{D\!N}^{-1}$ is a compact operator, $\Delta_{D\!N}^{-1} h_n \to \Delta_{D\!N}^{-1} h $ as $n\to +\infty$ strongly in $L^2_0({\rm D})$ and, by property \ref{positivesemidefinite} of $\tilde Z \subset Z$,
\begin{equation*}
    \langle \cH h_n, h_n \rangle = \| h_n \|^2 - \langle k^2 \Delta_{D\!N}^{-1} h_n, h_n \rangle \stackrel{n\to +\infty}{\longrightarrow} 1 - \langle k^2 \Delta_{D\!N}^{-1} h, h \rangle \geq \langle \cH h, h \rangle \geq 0.
\end{equation*}
Since $\langle \cH h_n, h_n \rangle \rightarrow 0$, we have $\| h \| = 1$ and $h \in K$ by Lemma \ref{lemZ}. By the Radon-Riesz property of the Hilbert space $L^2_0({\rm D})$, $h_n \to h$ as $n\to +\infty$ strongly. This concludes the proof of the claim. Since $\sL g = f$, 
\begin{equation}\label{hnid}
  h_n = \frac{f}{\| e^{t_n \sL} f \|}  + \sL \tilde g_n = o(1) + \sL  \tilde g_n ,\quad \text{with }\tilde g_n := \frac{e^{t_n\sL} g - g}{\| e^{t_n \sL} f \|} = \frac{\sL \int_0^{t_n} e^{s\sL}g\, {\rm d}s}{\| e^{t_n \sL} f \|} \in {\rm Ran}(\sL)  \subseteq \tilde Z.
\end{equation}
We split $\tilde g_n = k_n + g_n $, where $k_n \in {\rm K} \cap \tilde Z$ and $g_n \in {\rm K}^\perp \cap \tilde Z$.  It follows that  $ \sL g_n = \sL \tilde g_n$ and, by \eqref{coercivityH} and exploiting again the conservation of $\cH$ along trajectories,
$$
\begin{aligned}
c\| g_n \|^2 \leq \langle \cH g_n , g_n \rangle &= \langle \cH (\tilde g_n - k_n) ,  (\tilde g_n - k_n) \rangle = \langle \cH \tilde g_n,  \tilde g_n  \rangle \\
&\stackrel{\eqref{hnid}}{=} \frac{\langle \cH e^{t_n\sL} g,  e^{t_n\sL} g  \rangle}{ \| e^{t_n \sL} f \|^2 } + o(1)  =  \frac{\langle \cH g,   g  \rangle}{\| e^{t_n \sL} f \|^2 } + o(1) \to 0\quad \text{as }n\to +\infty.
\end{aligned}
$$
We have thus found a sequence $\{g_n\}_{n\in \N}\subset \tilde Z$ such that $\sL g_n = h_n + o(1)$ converges to $h\in {\rm K}$, with $\|h\|=1$. By the corollary result of Lemma \ref{banal_intersection} we conclude that $h=0$, which is a contradiction. 
\end{proof}
In view of the previous result, one is able to prove that a general initial datum in $L_{0}^{2}({\rm D})$ may grow at most linearly in time. Indeed, we know that this estimate is sharp, since, for example, $e^{t \sL}f_1^{\rm odd}=tf_0^{\rm odd}+f_1^{\rm odd} $.
\begin{lem}
   There exists a constant $C > 0$ such that, for every $f \in L_{0}^{2}({\rm D})$,
\begin{equation}\label{finalestimate}
        \| e^{t \sL} f \| \leq (1 + C \, t) \| f \| \qquad {\rm{for \, every}} \, t \geq 0.
    \end{equation}
\end{lem}
\begin{proof}
For every integer $n\geq 1$,
we write 
$    e^{n \sL} f - f = \sum_{j=0}^{n-1} e^{j \sL} (e^{\sL} f - f)
$,
with $e^{\sL} f - f =  \sL \int_{0}^{1} e^{s \sL} f {\rm d}s \in {\rm Ran}(\sL)$. By Lemma \ref{thanksGodforthislemma} and since, by standard energy estimates, $e^{\sL} - {\rm Id}$ is a bounded operator, there exists a positive constant $C_f$, depending on $f$, such that
\begin{equation}\label{secondestimate}
    \| e^{n \sL} f - f \| \leq \sum_{j=0}^{n-1} C_f \| e^{\sL} f - f \| \leq n C_f \|e^{\sL} - {\rm Id}\|   \|f\| .
\end{equation}
By the  uniform boundedness principle, $\| e^{n \sL} f - f \| \leq \tilde C n \|f\|$  for a constant $\tilde C > 0$ independent of $f$.
Let $t >0$ and split $t=n+r$ with integer $n \geq 0$, $ r \in [0,1)$. Again by standard energy estimates, we have 
$\| e^{r \sL} f \| \leq e^{M r} \| f \|$, for a constant $M>0$ independent of $f$, and, by exploiting the convexity of the exponential, $ e^{M r} \| f \| \leq (1 + e^{M} r) \| f \|$. In conclusion,
\begin{equation*}
    \| e^{t \sL} f \| = \| e^{r\sL} f + e^{r\sL} (e^{n\sL}-{\rm Id}) f \|   \leq \| e^{r \sL} f \| + \| e^{r \sL} (e^{n \sL} f - f ) \| \leq  (1 + C \, t) \| f \|  ,
\end{equation*}
for a constant $C > 0$ independent of $f$.
\end{proof}
We conclude our exposition with the following
\begin{proof}[Proof of Theorem \ref{lineardyn_Thm}]
Let us now consider the evolution $w(t)$ of an initial datum $w_0 \in L^1(\R^2) \cap L^p(\R^2)$, $p>2$, under the strongly continuous semigroup generated by $\cL_{eq}$ in \eqref{cLeq}. As described in Section \ref{reductiontodisk}, we have $w(t) = w_{ext}(t) + w_{int}(t)$, with $w_{ext}$ given by the pure transport equation \eqref{wext} and $w_{int}(t)= \widetilde w(t) + \mu {\bf 1}_{{\rm D}}$ where $\mu$ is a constant term and  the zero-average profile  $\widetilde w$ evolves according to the forced linear equation \eqref{complete-evolution}.  By Duhamel's formula,
    \begin{equation*}
        \widetilde w(t) = e^{t \cL_{eq}} \widetilde w_{0}(0) - \frac{\mu}{2} \int_{0}^{t} e^{(t-s) \cL_{eq}} f_{1}^{\rm ev} {\rm d}s - \int_{0}^{t} e^{(t-s) \cL_{eq}} \big( u[w_{ext}(s)] \cdot \nabla \omega_{eq} \big) {\rm d}s,
    \end{equation*}
Since, by Lemma \ref{generalizedkernel}, $\cL_{eq} f_1^{\rm ev} = -f_{0}^{\rm ev} \in {\rm Ker}(\cL_{eq})$, we have $e^{(t-s)\cL_{eq}} f_{1}^{ev} = f_{1}^{\rm ev} - (t-s) f_{0}^{\rm ev}$. One readily observes that the other forcing term $u[w_{ext}(s)] \cdot \nabla \omega_{eq}$ has zero average and is supported inside the unit disk ${\rm D}$. In view of Section \ref{passingtosL}, the linear evolution given by $e^{t \cL_{eq}} $ for initial data supported on ${\rm D}$ with zero average coincides with the one given by $e^{t \sL}$, hence
 \begin{equation}\label{Duhamelsolves}
      \widetilde w(t) = e^{t \sL} \widetilde w(0) - \frac{\mu}{2} t f_1^{\rm ev} + \frac{\mu}{4}  t^2 f_0^{\rm ev} - \int_{0}^{t} e^{(t-s) \sL} \big( u[w_{ext}(s)] \cdot \nabla \omega_{eq} \big) {\rm d}s.
 \end{equation}
By \eqref{finalestimate},
$
\| e^{t \sL} \widetilde w(0) \|\leq  (1 + C t) \| \widetilde w(0) \| $. Similarly, in view also of \cite[Remark 1.0.2]{afterVishik}, for every $0 \leq s \leq t
$,
$$
\begin{aligned}
&\| e^{(t-s) \sL} \big( u[w_{ext}(s)] \cdot \nabla \omega_{eq} \big) \|_{L^2} \leq \big(1 + C (t-s)\big) \|u[w_{ext}(s)] \cdot \nabla \omega_{eq}\|_{L^2} \\ &\leq C_p \big(1 + C (t-s)\big) \| \nabla \omega_{eq}\|_{L^\infty} \| w_{ext}(s) \|_{L^1 \cap L^p} \leq C_p \big(1 + C (t-s)\big) \| \nabla \omega_{eq}\|_{L^\infty} \| w_{ext}(0) \|_{L^1 \cap L^p},
\end{aligned}
$$
where the final inequality descends from the conservation of the $L^p$ norms, $p \in [1,+\infty]$, of \eqref{wext}. By the latter estimates and \eqref{Duhamelsolves},
we conclude the proof of \eqref{quadraticgrowth} and \eqref{quadraticgrowthinD}.
\end{proof}
\subsection*{Acknowledgements}
 We thank  M.~Colombo, M.~Dolce and G.~Cao-Labora  for
helpful discussions. P.~Ventura was supported by the Swiss State
Secretariat for Education, Research and Innovation (SERI) under contract number MB22.00034 through the project TENSE.

\end{document}